\documentclass[11pt,reqno]{amsart}
	\usepackage{amssymb}
	\usepackage{enumerate}
	\usepackage{tikz}
	\usepackage{graphicx}
	\usepackage{subcaption}

	\usetikzlibrary{math}
	\newtheorem{theorem}{Theorem}[section]

	\newtheorem{sublemma}{}[theorem]  
	\newtheorem{lemma}[theorem]{Lemma}
	\newtheorem{corollary}[theorem]{Corollary}
	\newtheorem{proposition}[theorem]{Proposition}
	
	\newtheorem{conjecture}[theorem]{Conjecture}

	\newcommand{\cM}{\mathcal{M}}

	\DeclareMathOperator{\si}{si}
	
	\DeclareMathOperator{\cl}{cl}

	\newcommand{\del}{\setminus}
	\newcommand{\con}{/}

\begin{document}
\sloppy
 \begin{abstract}
For a  matroid $M$ having $m$ rank-one flats, the density $d(M)$ is $\tfrac{m}{r(M)}$ unless $m = 0$, in which case $d(M)= 0$. 
A matroid is density-critical if all of its proper minors of non-zero rank have lower density. By a 1965 theorem of Edmonds, a matroid that 
is minor-minimal among simple matroids that cannot be covered by $k$ independent sets is density-critical. 
 It is straightforward to show that  $U_{1,k+1}$ is the only minor-minimal loopless matroid with no covering by 
  $k$ independent sets. We prove that there are exactly 
 ten minor-minimal simple obstructions to a matroid being able to be covered by two independent sets. These ten matroids are precisely the density-critical matroids $M$ such that $d(M) > 2$ but $d(N) \le 2$ for all proper minors $N$ of $M$.  All   density-critical matroids of density less than $2$ are series-parallel networks.  For $k \ge 2$, although finding all density-critical matroids of density at most $k$ does not seem straightforward,  
we do solve this problem for  $k=\tfrac{9}{4}$. 
 
 \end{abstract}

\title[On density-critical matroids]{On density-critical matroids}

	\author[R. Campbell]{Rutger Campbell}
	\address{Department of Combinatorics and Optimization,
	University of Waterloo, Waterloo, Canada}
	\email{rtrcampb@uwaterloo.ca}
	\author[K. Grace]{Kevin Grace}
\address{School of Mathematics, University of Bristol, Bristol, UK}
\email{kevin.grace@bristol.ac.uk}	
	\author[J. Oxley]{James Oxley}
\address{Department of
 Mathematics, Louisiana State University, Baton Rouge, Louisiana, USA}
\email{oxley@math.lsu.edu}
\author[G. Whittle]{Geoff Whittle}
\address{School of Mathematics, Statistics and Operations Research, Victoria University of Wellington, Wellington, New Zealand}
\email{geoff.whittle@vuw.ac.nz}
	
	\date{\today}
\maketitle

\section{Introduction}
Our notation and terminology  follow Oxley~\cite{JGO11}. For a positive integer $k$, 
let $\cM_k$ be the class of matroids $M$ for which $E(M)$ is the union of  
$k$  independent sets. We say such a matroid can be {\em covered} by $k$  independent sets.
Edmonds \cite{JE65} gave  the following characterization of the members of $\cM_k$.

\begin{theorem}\label{edmonds}
	A matroid $M$ has $k$ independent sets whose union is $E(M)$ if and only if, for every subset $A$ of $E(M)$, 
		$$k\,r(A) \ge |A|.$$ 
\end{theorem}

Clearly, $\cM_k$ is closed under deletion. However, $\cM_k$ is not closed under contraction.
For example, the $6$-element rank-$3$ uniform matroid $U_{3,6}$ can be covered by two independent sets, yet contracting a point of this matroid
gives $U_{2,5}$, which cannot.  
For all $k$, 
the loop is the unique minor-minimal matroid not in $\cM_k$.
On that account, we limit the types of obstructions we consider.
We first examine the minor-minimal loopless matroids that are not in $\cM_k$. We find the following result.

\begin{proposition}
\label{u1k}
	The unique minor-minimal loopless matroid that cannot be covered by $k$ independent sets is $U_{1,k+1}$. 
	\end{proposition}

Restricting  attention to minor-minimal simple matroids not in $\cM_k$, we find much more structure.
We have the following collection of ten matroids for the case when $k$ is two. In this result, $P(M_1,M_2)$ denotes the parallel connection of matroids $M_1$ and $M_2$, this matroid being unique when both $M_1$ and $M_2$ have transitive automorphism groups.  Geometric representations of the nine of these ten matroids of rank at most four are shown in Figure~1. A diagram representing the tenth matroid, $P(M(K_4),M(K_4))$ is also given where we note that this matroid has rank five.

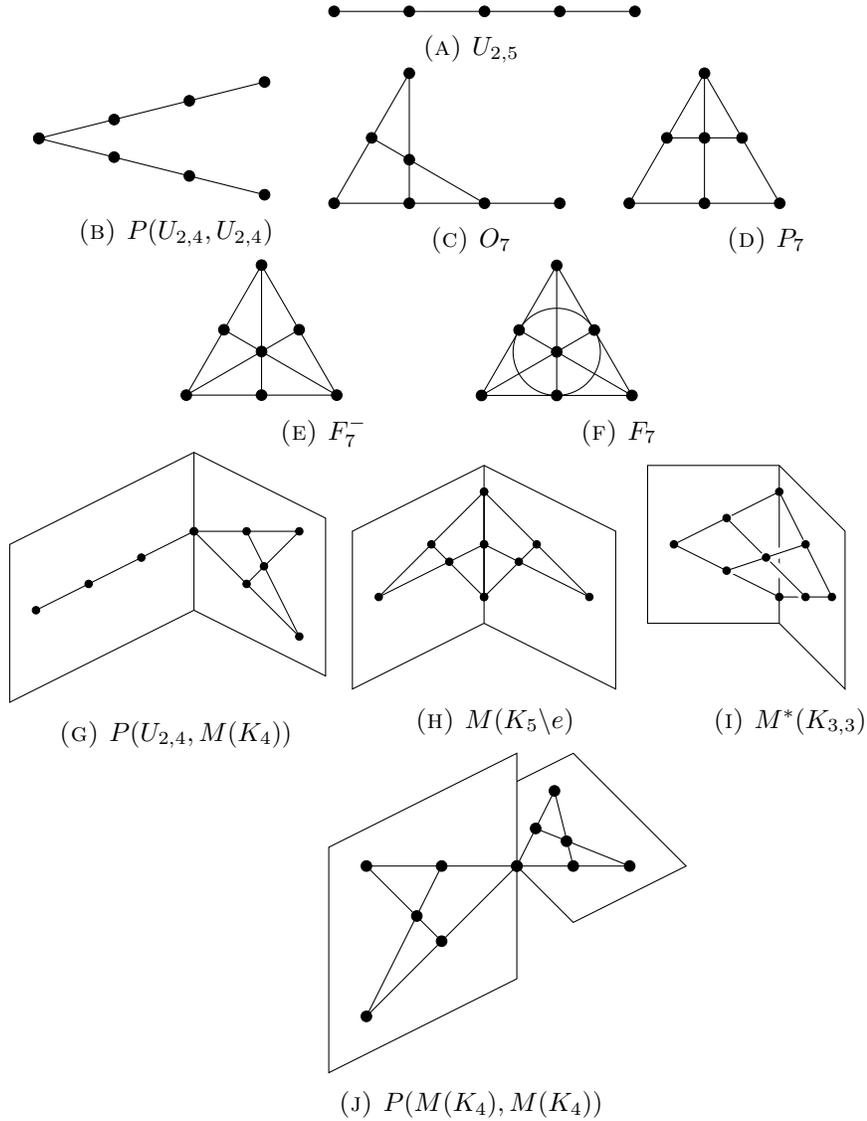
\begin{figure}
\centering
	\begin{subfigure}{0.3\textwidth}
		\begin{tikzpicture}
			\filldraw (0,0) circle (2pt);
			\filldraw (1,0) circle (2pt);
			\filldraw (2,0) circle (2pt);
			\filldraw (3,0) circle (2pt);
			\filldraw (4,0) circle (2pt);
			\draw (0,0) -- (4,0);
		\end{tikzpicture}
		\caption{$U_{2,5}$}
	\end{subfigure}

	\begin{subfigure}{0.3\textwidth}
		\begin{tikzpicture}
			
			\filldraw (0,0) circle (2pt);
			\filldraw (1,-0.25) circle (2pt);
			\filldraw (2,-0.5) circle (2pt);
			\filldraw (3,-.75) circle (2pt);
			\draw (0,0) -- (3,-0.75);
			
			\filldraw (0,0) circle (2pt);
			\filldraw (1,0.25) circle (2pt);
			\filldraw (2,0.5) circle (2pt);
			\filldraw (3,0.75) circle (2pt);
			\draw (0,0) -- (3,0.75);
		\end{tikzpicture}
		\caption{$P(U_{2,4},U_{2,4})$}
	\end{subfigure}
	\begin{subfigure}{0.3\textwidth}
		\begin{tikzpicture}
			\filldraw (0,0) circle (2pt);
			\filldraw (1,0) circle (2pt);
			\filldraw (2,0) circle (2pt);
			\filldraw (3,0) circle (2pt);
			\filldraw (0.5,0.87) circle (2pt);
			\filldraw (1,1.73) circle (2pt);
			\filldraw (1,0.58) circle (2pt);
			\draw (0,0) -- (3,0);
			\draw (0,0) -- (1,1.73);
			\draw (1,1.73) -- (1,0);
			\draw (0.5,0.87) -- (2,0);
		\end{tikzpicture}
		\caption{$O_7$}
	\end{subfigure}
	\begin{subfigure}{0.3\textwidth}
		\begin{tikzpicture}
			\filldraw (0,0) circle (2pt);
			\filldraw (1,0) circle (2pt);
			\filldraw (2,0) circle (2pt);
			\filldraw (0.5,0.87) circle (2pt);
			\filldraw (1,0.87) circle (2pt);
			\filldraw (1.5,0.87) circle (2pt);
			\filldraw (1,1.73) circle (2pt);
			\draw (0,0) -- (2,0);
			\draw (0.5,0.87) -- (1.5,0.87);
			\draw (1,1.73) -- (0,0);
			\draw (1,1.73) -- (1,0);
			\draw (1,1.73) -- (2,0);
		\end{tikzpicture}
		\caption{$P_7$}
	\end{subfigure}
	\begin{subfigure}{0.3\textwidth}
		\begin{tikzpicture}
			\filldraw (0,0) circle (2pt);
			\filldraw (1,0) circle (2pt);
			\filldraw (2,0) circle (2pt);
			\filldraw (0.5,0.87) circle (2pt);
			\filldraw (1,0.58) circle (2pt);
			\filldraw (1.5,0.87) circle (2pt);
			\filldraw (1,1.73) circle (2pt);
			\draw (0,0) -- (2,0);
			\draw (0.5,0.87) -- (2,0);
			\draw (0,0) -- (1.5,0.87);
			\draw (1,1.73) -- (0,0);
			\draw (1,1.73) -- (1,0);
			\draw (1,1.73) -- (2,0);
		\end{tikzpicture}
		\caption{$F^-_7$}
	\end{subfigure}
	\begin{subfigure}{0.3\textwidth}
		\begin{tikzpicture}
			\filldraw (0,0) circle (2pt);
			\filldraw (1,0) circle (2pt);
			\filldraw (2,0) circle (2pt);
			\filldraw (0.5,0.87) circle (2pt);
			\filldraw (1,0.58) circle (2pt);
			\filldraw (1.5,0.87) circle (2pt);
			\filldraw (1,1.73) circle (2pt);
			\draw (0,0) -- (2,0);
			\draw (0.5,0.87) -- (2,0);
			\draw (0,0) -- (1.5,0.87);
			\draw (1,1.73) -- (0,0);
			\draw (1,1.73) -- (1,0);
			\draw (1,1.73) -- (2,0);
			\draw (1,0.58) circle (0.58);
		\end{tikzpicture}
		\caption{$F_7$}
	\end{subfigure}

	\begin{subfigure}{0.35\textwidth}
		\begin{tikzpicture}[scale=0.7]
			\tikzmath{\S=0.5;}
		    \draw (2.5,-1.5-\S*2.5) -- (2.5,1.5-\S*2.5) -- (0,1.5) -- (-3.5,1.5-\S*3.5) -- (-3.5,-1.5-\S*3.5) -- (0,-1.5) -- cycle;
			\draw (0,-1.5) -- (0,1.5);
			
			\filldraw (0,0) circle (2pt);
			\filldraw (1,0.5-\S*1) circle (2pt);
			\filldraw (1,-0.5-\S*1) circle (2pt);
			\filldraw (1.33,0-\S*1.33) circle (2pt);
			\filldraw (2,1-\S*2) circle (2pt);
			\filldraw (2,-1-\S*2) circle (2pt);
			\draw (0,0) -- (2,1-\S*2);
			\draw (0,0) -- (2,-1-\S*2);
			\draw (1,-0.5-\S*1) -- (2,1-\S*2);
			\draw (1,0.5-\S*1) -- (2,-1-\S*2);
			
			\filldraw (0,0) circle (2pt);
			\filldraw (-1,0-\S*1) circle (2pt);
			\filldraw (-2,0-\S*2) circle (2pt);
			\filldraw (-3,0-\S*3) circle (2pt);
			\draw (0,0) -- (-3,0-\S*3);
		\end{tikzpicture}
		\caption{$P(U_{2,4},M(K_4))$}
	\end{subfigure}
	\begin{subfigure}{0.3\textwidth}
		\begin{tikzpicture}[scale=0.7]
			\tikzmath{\S=0.5;}
		    \draw (-2.5,-1.5-\S*2.5) -- (-2.5,1.5-\S*2.5) -- (0,1.5) -- (2.5,1.5-\S*2.5) -- (2.5,-1.5-\S*2.5) -- (0,-1.5) -- cycle;
			\draw (0,-1.5) -- (0,1.5);
			
			\filldraw (0,0) circle (2pt);
			\filldraw (0,-1) circle (2pt);
			\filldraw (0,1) circle (2pt);
			\filldraw (-2,0-\S*2) circle (2pt);
			\filldraw (-1,0.5-\S*1) circle (2pt);
			\filldraw (-0.66,0-\S*0.66) circle (2pt);
			\draw (0,1) -- (0,-1);
			\draw (0,1) -- (-2,0-\S*2);
			\draw (0,0) -- (-2,0-\S*2);
			\draw (0,-1) -- (-1,0.5-\S*1);
			
			\filldraw (0,0) circle (2pt);
			\filldraw (0,-1) circle (2pt);
			\filldraw (0,1) circle (2pt);
			\filldraw (2,0-\S*2) circle (2pt);
			\filldraw (1,0.5-\S*1) circle (2pt);
			\filldraw (0.66,0-\S*0.66) circle (2pt);
			\draw (0,1) -- (0,-1);
			\draw (0,1) -- (2,0-\S*2);
			\draw (0,0) -- (2,0-\S*2);
			\draw (0,-1) -- (1,0.5-\S*1);
		\end{tikzpicture}
		\caption{$M(K_5\backslash e)$}
	\end{subfigure}
	\begin{subfigure}{0.3\textwidth}
		\begin{tikzpicture}[scale=0.7]
			\tikzmath{\S=0.5;}
		    \draw (-2.5,-1.5) -- (-2.5,1.5) -- (0,1.5) -- (1.25,1.5-\S*2.5) -- (1.25,-1.5-\S*2.5) -- (0,-1.5) -- cycle;
			\draw (0,-1.5) -- (0,1.5);

			\draw (-2,0) -- (0,-1) -- (1,-1) -- (0,1) -- cycle;
			\draw [white, line width=4pt]  (-1,0.5) -- (0.5,-1);
			\draw (-1,0.5) -- (0.5,-1);
			\draw [white, line width=4pt]  (-1,-0.5) -- (0.5,0);
			\draw (-1,-0.5) -- (0.5,0);
			
			\filldraw (0,-1) circle (2pt);
			\filldraw (0,1) circle (2pt);
			\filldraw (-2,0) circle (2pt);
			\filldraw (-1,0.5) circle (2pt);
			\filldraw (-1,-0.5) circle (2pt);
			\filldraw (1,-1) circle (2pt);
			\filldraw (0.5,-1) circle (2pt);
			\filldraw (0.5,0) circle (2pt);

			\filldraw (-0.25,-0.25) circle (2pt);
			
		\end{tikzpicture}
		\caption{$M^*(K_{3,3})$}
	\end{subfigure}

	\begin{subfigure}{0.3\textwidth}
		\begin{tikzpicture}
			\tikzmath{\S=0.5;}
		    \draw (-2.5,-1.5-\S*2.5) -- (-2.5,1.5-\S*2.5) -- (0,1.5) -- (0,-1.5) -- cycle;
		    \draw (0,0.75+\S*0.75) -- (0.75,0.75+\S*1.5) -- (2.25,-0.75+\S*1.5) -- (0.75,-0.75) -- (0,0);
			
			\filldraw (0,0) circle (2pt);
			\filldraw (-1,0.5-\S*1) circle (2pt);
			\filldraw (-1,-0.5-\S*1) circle (2pt);
			\filldraw (-1.33,0-\S*1.33) circle (2pt);
			\filldraw (-2,1-\S*2) circle (2pt);
			\filldraw (-2,-1-\S*2) circle (2pt);
			\draw (0,0) -- (-2,1-\S*2);
			\draw (0,0) -- (-2,-1-\S*2);
			\draw (-1,-0.5-\S*1) -- (-2,1-\S*2);
			\draw (-1,0.5-\S*1) -- (-2,-1-\S*2);

			\filldraw (0,0) circle (2pt);
			\filldraw (0.25,0.25+\S*0.5) circle (2pt);
			\filldraw (0.75,-0.25+\S*0.5) circle (2pt);
			\filldraw (0.66,\S*0.66) circle (2pt);
			\filldraw (0.5,0.5+\S*1) circle (2pt);
			\filldraw (1.5,-0.5+\S*1) circle (2pt);
			\draw (0,0) -- (0.5,1.5-\S*1);
			\draw (0,0) -- (1.5,0.5-\S*1);
			\draw (0.75,0.25-\S*0.5) -- (0.5,1.5-\S*1);
			\draw (0.25,0.75-\S*0.5) -- (1.5,0.5-\S*1);
			


		\end{tikzpicture}
		\caption{$P(M(K_4), M(K_4))$}
	\end{subfigure}
	\caption{The minor-minimal simple matroids not in $\cM_2$}
\end{figure}

\begin{theorem}
\label{mainone}
	The minor-minimal simple matroids that cannot be covered by two independent sets are 
	$U_{2,5}$, $P(U_{2,4},U_{2,4})$, $O_7$, $P_7$, $F_7^-$, $F_7$, $P(U_{2,4},M(K_4))$,   $M(K_5\del e)$,  $M^*(K_{3,3})$, and $P(M(K_4),M(K_4))$.		
\end{theorem}

The following consequence of Theorem~\ref{edmonds} will be helpful.

\begin{lemma}
\label{*}
Let $M$ be a minor-minimal matroid that cannot be covered by $k$ independent sets. Then 
$$k\,r(M) = |E(M)| - 1.$$
Moreover, $M$ has no coloops.
\end{lemma}

For a matroid $M$, we write $\varepsilon(M)$ for $|E(\si(M))|$, the number of rank-one flats of $M$. The \emph{density} $d(M)$ of  $M$  is $\tfrac{\varepsilon(M)}{r(M)}$ unless $r(M) = 0$.    In the exceptional case,  $\varepsilon(M) = 0$ and  we define $d(M) = 0$. We say that $M$ is {\em density-critical} when $d(N) < d(M)$ for all proper minors $N$ of $M$. Note that all density-critical matroids are simple.
By Lemma~\ref{*} and Theorem~\ref{edmonds},  $M$ is a minor-minimal simple matroid that cannot be covered by $k$ independent sets if and only if $d(M) >k$ but $d(N) \le k$ for all proper minors $N$ of $M$. Such matroids are strictly $k$-density-critical where, for $t\geq 0$,
 we say a matroid is \emph{strictly  $t$-density-critical}
 when its density is  strictly  greater than $t$ 
while all its proper minors have density at most $t$. Thus Theorem~\ref{mainone} explicitly determines all ten strictly $2$-density-critical matroids. 

We propose the following. 

\begin{conjecture}\label{fin_obst_conj}
	For all positive integers $k$, there are finitely many minor-minimal simple matroids that cannot be covered by $k$ independent sets.
\end{conjecture}

More generally, we make  the following conjectures. For $t> 0$,
 we say a matroid is \emph{$t$-density-critical}
 when its density is  at least $t$ 
while all of its proper minors have density strictly less than $t$.

\begin{conjecture}\label{fin_tdc_conj}
	For all $t\geq 0$, there are finitely many strictly $t$-density-critical matroids.
\end{conjecture}

\begin{conjecture}\label{fin_tdc_conj2}
	For all $t>0$, there are finitely many  $t$-density-critical matroids.
\end{conjecture}

We  also propose the following weakening of the last conjecture.

\begin{conjecture}\label{fin_dc_conj}
	For all $t\geq 0$, there are finitely many density-critical matroids with density exactly $t$.
\end{conjecture}

We note that these conjectures hold over any class of matroids that is well-quasi-ordered with respect to minors.
In particular, by a  result announced by Geelen, Gerards, and Whittle (see, for example, \cite{GGW}), these
conjectures hold within the class of matroids representable over a fixed finite field.

Because the two excluded minors for series-parallel networks, $U_{2,4}$ and $M(K_4)$, have density exactly two, for $k < 2$, all   density-critical matroids of density at most $k$ are series-parallel networks.  For $k > 2$, finding all density-critical matroids of density at most $k$ does not seem straightforward.
However, we were able to solve this problem when $k= \tfrac{9}{4}$. For all $n \ge 2$, we denote by $P_n$   any matroid that can be constructed from $n$ copies of $M(K_3)$ via a sequence of $n-1$ parallel connections. In particular, $P_2 \cong M(K_4 \backslash e)$. There are two choices for $P_3$ depending on which element of $M(K_4 \backslash e)$ is used as the basepoint of the parallel connection with the third copy of $M(K_3)$.   We denote by $M_{18}$  the $18$-element matroid that is obtained by attaching, via parallel connection, a copy of $M(K_4)$  at each element of an $M(K_3)$.

\begin{theorem}
\label{9/4}
The following is a list of all pairs $(M,d)$ where $M$ is a density-critical matroid of density $d$ and    $d \le \tfrac{9}{4}$: \\
$(U_{1,1},1)$, $(U_{2,3},\tfrac{3}{2})$,  
$(M\left(P_n\right),\tfrac{2n+1}{n+1})$ for all $n \ge 2$, $(U_{2,4}, 2)$, 
 $(M(K_4),2)$, $(P(M(K_4),M(K_4)),\tfrac{11}{5})$, $(P(U_{2,4}, M(K_4)),\tfrac{9}{4})$, $(M(K_5\backslash e),\tfrac{9}{4})$,  $(M^*(K_{3,3}),\tfrac{9}{4})$, $(M_{18},\tfrac{9}{4})$.  
\end{theorem}


\section{Preliminaries}
\label{prelim}

This section proves some preliminary results beginning with  two that were stated in the introduction. 

\begin{proof}[Proof of Proposition~\ref{u1k}.]
Clearly, $U_{1,k+1}$ is a minor-minimal loopless matroid that cannot be covered by $k$ independent sets. Conversely, suppose that $M$ is a minor-minimal loopless matroid that cannot be covered by $k$ independent sets.
Certainly, $M$ contains some element $e$. 
    Let $P\cup \{e\}$  be the parallel class of $M$ that contains $e$ where $P = \{e_1,e_2,\dots,e_\ell\}$ and $e \not \in P$. 
    Now $M\con e \del P$ is loopless, so, by minimality, $M\con e \del P$ can be covered by $k$ independent sets $\{A_1,A_2,\dots,A_k\}$. 
    Note that each $A_i\cup\{e\}$ is independent in $M$,
    so if $|P|=\ell\leq k-1$, then $\{A_1\cup\{e_1\},A_2\cup\{e_2\},\dots,A_\ell\cup\{e_\ell\},A_{\ell+1}\cup \{e\},\dots,A_k\cup \{e\}\}$  is a set of $k$ independent sets that covers $M$.  
    Thus $|P|\geq k$, and so $M \cong U_{1,k+1}$.
\end{proof}

Since $U_{1,k+1}$ is a $(k+1)$-element cocircuit,   the    matroids having no $U_{1,k+1}$-minor are precisely the   matroids for which every cocircuit has at most $k$ elements.

\begin{proof}[Proof of Lemma~\ref{*}.]
Take $x$ in $E(M)$. Then $M\del x$ can be covered by $k$ independent sets. Thus, by Theorem~\ref{edmonds},  
$$|E(M)| > k r(M) \ge kr(M\del x) \ge |E(M\del x)| = |E(M)| - 1.$$ 
We deduce that $kr(M) = |E(M)| - 1$ and $r(M) = r(M\del x)$ so $M$ has no coloops. 
\end{proof}

\begin{lemma}
\label{lem0}
Let $M$ be a density-critical matroid of rank at least two. For each subset $S$ of $E(M)$, 
$$|E(M)| - \varepsilon(M/S) > d(M)r(S).$$
In particular, every element of $M$ is in a triangle and is in at least two triangles when $d(M) \ge 2$. 
\end{lemma}

\begin{proof}
Since $M$ is density-critical and therefore simple, 
$$\frac{\varepsilon(M/S)}{r(M/S)} <\frac{\varepsilon(M)}{r(M)}= \frac{|E(M)|}{r(M)}.$$ 
Hence $r(M) \varepsilon(M/S) < |E(M)|(r(M) - r(S))$, so
$$r(M)d(M)r(S) = |E(M)|r(S) < r(M)\left(|E(M)| - \varepsilon(M/S)\right).$$
Thus $d(M)r(S) < |E(M)| - \varepsilon(M/S)$. In particular,   $d(M)  < |E(M)| - \varepsilon(M/e)$ for all $e$ in $E(M)$. 
Hence   every such element $e$ is  in at least one triangle, and $e$ is in at least two triangles   when 
$d(M) \ge 2$.
\end{proof}

The next result will be useful in the proof of Theorem~\ref{mainone}.

\begin{lemma}
\label{ww0}
In a $3$-connected matroid $M$, let $F$ be a $2k$-element set $\{b_1,a_1,b_2,a_2,\dots,b_k,a_k\}$. Suppose $\{b_1,b_2,\dots,b_k\}$ is independent and $\{b_i,a_i,b_{i+1}\}$ is a circuit for all $i$, where $b_{k+1} = b_1$. Then $M|F$ is a wheel of rank at least three or a whirl of rank at least two.
\end{lemma}

\begin{proof}
Since $M$ is $3$-connected with at least four elements, it is simple. Now $M|F$ has $\{a_i,b_{i+1},a_{i+1}\}$ as a triad, where $a_{k+1} = a_1$. By a result of Seymour \cite{PDS80} (see also \cite[Lemma 8.8.5(ii)]{JGO11}), $M|F$ is a wheel or a whirl of rank $k$. 
\end{proof}

\section{The  matroids that cannot be covered by two independent sets}

In this section, we prove Theorem~\ref{mainone}  thereby specifying all of the minor-minimal simple matroids that cannot be covered by two independent sets. 

\setcounter{theorem}{1}

\begin{proof}[Proof of Theorem~\ref{mainone}.]
It is straightforward to check that each of the matroids listed is a minor-minimal simple matroid that cannot be covered by two independent sets. 
Now let $M$ be such a matroid. The next two assertions are immediate consequences of  Lemmas~\ref{*}, \ref{lem0}, and \ref{edmonds}. 
However, we include proofs independent of Edmonds's result for completeness.
	\begin{sublemma}\label{triangle_claim}
		Every element of $M$ is contained in at least two triangles.
	\end{sublemma}
			Let $e$ be an element of $M$ and let $M' = \si(M/e)$. 
			By minimality, $M'$  has a partition into two independent sets $A$ and $B$.
			Suppose $e$ is not in a triangle. Then $E(M')=E(M) -\{e\}$ and we have $r_M(A\cup \{e\})=r_{M'}(A)+1=|A|+1$ and $r_M(B\cup\{e\})=|B|+1$,  
			so $M$ is covered by the independent sets $A\cup \{e\}$ and $B\cup \{e\}$, which is a contradiction.
			
			Now suppose $e$ is in exactly one triangle $\{e,c,d\}$ of $M$.
			We may assume that $M' = M/e\backslash c$ and that $d \in A$. Then 
			 $r_M(A\cup \{c\})= r_M(A \cup \{c,e\}) = r_{M'}(A)+1=|A|+1$ and $r_M(B\cup \{e\})=r_{M'}(B)+ 1=|B|+1$, 
			so $M$ is covered by the independent sets $A\cup \{c\}$ and $B\cup\{e\}$. This contradiction implies that 
			 \ref{triangle_claim} holds.
	
	\begin{sublemma}\label{size_claim}
		$|E(M)| \leq 2r(M) + 1$ and $|A|\leq 2r(A)$ for every proper subset $A$ of $E(M)$.
	\end{sublemma}
			Suppose $A$ is a proper subset of $E(M)$. 
			 By the minimality of $M$,
			we can cover $M|A$ by two independent sets,
			and so $|A|\leq 2 r(A)$.
			It  follows easily that
			$|E(M)|\leq 2 r(M)+1$. 
			Thus~\ref{size_claim} holds.

	We construct a simple auxiliary graph $G$ from $M$, the vertices of which are the elements of $M$; two such vertices are adjacent exactly 
	when they share a triangle in $M$.
	Next, we show the following.
	\begin{sublemma}
	\label{ww}
		Let $Z$ be the vertex set of a component of $G$. Then $M|Z$  has a wheel or a whirl as a restriction.
	\end{sublemma}
	
	We may assume that $M|Z$ has no line with four or more points otherwise $M$ has a rank-$2$ whirl as a restriction. For $b_1$ in $Z$, 
	by \ref{triangle_claim}, we can construct a maximal sequence $b_1,a_1,b_2,a_2,\dots,b_n$ of distinct elements such that $\{b_1,b_2,\dots,b_n\}$ is independent and 
	$\{b_i,a_i,b_{i+1}\}$ is a triangle for all $i$ in $\{1,2,\dots,n-1\}$. Then $n \ge 3$. 
	
	Now $M$ has a triangles $\{b_n,a_n,b_{n+1}\}$ and $\{b_0,a_0,b_1\}$ that differ  from 
	$\{b_{n-1},a_{n-1},b_{n}\}$ and $\{b_1,a_1,b_2\}$, respectively.  Let $A' = \{b_1,a_1,b_2,a_2,\dots,b_{n-1},a_{n-1},b_n\}$. 
	 Assume that both $\{a_n, b_{n+1}\}$ and $\{a_0,b_0\}$ avoid $A'$. Then 
	 $|A' \cup \{a_n, b_{n+1}\}| = 2n+1 = 2r(A' \cup \{a_n, b_{n+1}\}) + 1$. Thus, 
	 by~\ref{size_claim}, $A' \cup \{a_n, b_{n+1}\} = E(M)$. By symmetry, $A' \cup \{a_0, b_{0}\} = E(M)$. Hence 
	  $\{a_n,b_{n+1}\}=\{b_0,a_0\}$, so $\{b_n,a_n,b_{n+1},b_1\}$ is a $4$-point line, a contradiction.
		
		We may now assume that   $b_{n+1}$ is a member $c_i$ of $\{b_i,a_i\}$ for some $i$ with $1 \le i \le n-1$. Then $\{c_i,b_{i+1},b_{i+2},\dots,b_{n}\}$ is an independent set in $M|Z$ such that every two consecutive elements in the given cyclic order  are in a triangle. Thus, by Lemma~\ref{ww0}, $M|Z$ has a wheel or whirl of rank $n-i+1$ as a restriction.  Hence \ref{ww} holds.

	 \begin{sublemma}
	 \label{nottoo}
	 For some component of $G$ having vertex set $Z$, the matroid $M|Z$ is not a wheel or a whirl.
	 \end{sublemma}
	 
	 Assume that this fails. Then, by \ref{triangle_claim}, the only components of $G$ are rank-$2$ whirls or rank-$3$ wheels. Assume there are $s$ of the former and $t$ of the latter. 
	 Then
	    $|E(M)| = 4s+ 6t = 2(2s + 3t)$. Clearly $r(M) \le 2s + 3t$. 
	 By~\ref{size_claim}, equality must hold  here. Hence each component of $G$ corresponds to a wheel or whirl component of $M$. As each wheel and each whirl can be covered by     two independent sets, so too can $M$,  
	  a contradiction. Thus~\ref{nottoo} holds.

		Now take a component of $G$ having vertex set $Z$ such that $M|Z$ is not a wheel or a whirl.
		By~\ref{ww}, consider a wheel or whirl restriction of $M|Z$  with basis $B=\{b_1,b_2,\dots,b_n\}$ and ground set $W=\{b_1,a_1,b_2,a_2,\dots,b_n,a_n\}$. Let 
		  $\{b_i,a_i,b_{i+1}\}$ be a triangle for all $i$ where $b_{n+1} = b_1$. As $W \neq Z$, 
		 there is a point $\beta_1$ in $W$ that is contained in a triangle $\{\beta_1,\alpha_1,\beta_2\}$ that is not a triangle of $M|W$. 
		If $M|W$ is a rank-$2$ whirl or a rank-$3$ wheel, then, by symmetry, we may assume that $\beta_1=a_1$. 
		If, instead,  $M|W$ is neither a rank-$2$ whirl nor a rank-$3$ wheel,
		then \ref{triangle_claim} guarantees that  such a triangle $\{\beta_1,\alpha_1,\beta_2\}$   exists with $\beta_1=a_1$.
		By repeatedly using  \ref{triangle_claim},
		we can construct a  sequence     $\beta_1,\alpha_1,\dots, \beta_{m+1}$ where $\{\beta_i,\alpha_i,\beta_{i+1}\}$ is a triangle for all $i$ in $\{1,2, \dots, m\}$ and  
		 $B\cup \{\beta_2,\dots,\beta_{m+1}\}$ is  dependent but 
		 $B\cup\{\beta_2,\dots,\beta_m\}$ is independent. By potentially interchanging $\alpha_m$ and $\beta_{m+1}$,   we may assume that $\alpha_m\notin W$. Let $Q = \{\beta_1,\alpha_1,\dots, \beta_{m+1}\}$. Then 
		 \begin{equation}
		 \label{star}
		 r(W \cup Q) = r(W \cup (Q - \{\beta_{m+1}\})) = n+m -1.
		 \end{equation}
		 As $|W \cup (Q - \{\beta_{m+1}\})| = 2(n+m -1) + 1 = 2r(W \cup (Q - \{\beta_{m+1}\})) + 1$, we deduce, by \ref{size_claim}, that 
		  \begin{equation}
		 \label{*em}
		 W \cup (Q - \{\beta_{m+1}\}) = E(M). 
		 \end{equation}
		 Hence 
		\begin{equation}
		\label{beta}
		\beta_{m+1} \in W \cup (Q - \{\beta_{m+1}\}).
		\end{equation}

	 Assume that the theorem fails. We now show that  
	\begin{sublemma}\label{wh_claim}
		$M|Z$ has no wheel-restriction of rank exceeding three  and no whirl-restriction of rank exceeding two. 
			\end{sublemma}
	
		Assume that this fails. Then we may assume that $M|W$ is a wheel of rank at least four or a whirl of rank at least three. 
		Now $r(W) = n$ and $r(Q) \le m+1$. By (\ref{star}) and submodularity, $r(\cl(W) \cap \cl(Q))\le 2$. Assume $W$ does not span $M$. Then, by (\ref{star}) and (\ref{*em}), we see that $m > 1$ and the only possible elements of $W$ that can lie in triangles with elements of $Q - W$ are $\beta_1$ and $\beta_{m+1}$. But a wheel of rank at least four and a whirl of rank at least three have at least three elements that are in unique triangles. Hence one of these elements will violate \ref{triangle_claim}. 
		
		We now know that $W$ spans $M$, so the unique element of $Q - W$ is $\alpha_1$. Each of $a_1,a_2,\dots,a_n$ must be in a triangle with $\alpha_1$, the other element of which is in $W$. Assume both $\{a_1,\alpha_1,a_3\}$ and $\{a_1,\alpha_1,a_{n-1}\}$ are triangles. Then $n = 4$. Suppose $\{a_2,\alpha_1,a_4\}$ is also a triangle. Then, by Lemma~\ref{ww0}, for each $i$ in $\{2,4\}$, deleting   $a_i$  from $M|(W \cup Q)$ gives a wheel or whirl of rank four. As $\{b_1,b_4,\alpha_1,a_2\}$ and $\{b_2,b_3,\alpha_1,a_4\}$ are circuits, both of these deletions are wheels. It follows that 
$M|(W \cup Q) \cong M^*(K_{3,3})$, so $M \cong M^*(K_{3,3})$, a contradiction. Thus, we may assume that $\{a_2,\alpha_1,a_4\}$ is not a triangle. Since $\alpha_1 \not\in \cl(\{b_1,b_2,b_3\}) \cup \cl(\{b_2,b_3,b_4\})$, there is no triangle containing $\{a_2,\alpha_1\}$, a contradiction. 
		
		We  may now assume that $\{a_1,\alpha_1,a_3\}$ is not a triangle. Then, by~\ref{triangle_claim},
		$W$ has distinct elements $x$  and $y$   such that $\{a_1,\alpha_1,x\}$ and $\{a_3,\alpha_1,y\}$ are triangles.
		Thus $\{a_1,a_3,x,y\}$ contains a circuit. Now $\{a_1,a_3\}$ is not in a triangle of $M|W$. Moreover, if  $\{a_1,x,y\}$ is a triangle, then $\{x,y\} = \{b_1,b_2\}$. Using the triangles,  $\{a_1,\alpha_1,x\}$ and $\{a_3,\alpha_1,y\}$, we deduce that  $a_3 \in \cl(\{b_1,b_2\})$, a contradiction. It follows that $\{a_1,a_3,x,y\}$ is a circuit of $M$. Thus   $M|W$ is either a rank-$3$ whirl or a rank-$4$ wheel.  
		
		Suppose $M|W$ is a rank-$3$ whirl. Then $M$ is an extension of this matroid by $\alpha_1$ in which every element is in at least two triangles. If $\{a_1,a_2,\alpha_1\}$ or $\{a_2,a_3,\alpha_1\}$ is a triangle, then one easily checks that $M \cong O_7$ or $M \cong P_7$, a contradiction. Hence we may assume that none of  $\{a_1,a_2,\alpha_1\}$, $\{a_2,a_3,\alpha_1\}$, or  $\{a_3,a_1,\alpha_1\}$ is a triangle. Then, to avoid having $U_{2,5}$ as a minor of $M$, we must have $\{a_1,b_3,\alpha_1\}$, $\{a_2,b_1,\alpha_1\}$, and $\{a_3,b_2,\alpha_1\}$ as triangles, that is, $M \cong F^- _7$, a contradiction. 
		
		We are left with the possibility that $M|W$ is a rank-$4$ wheel. Since it has $\{a_1,a_3,x,y\}$ as a circuit, it follows that $\{x,y\} = \{a_2,a_4\}$. Then   $M$ has either $\{a_1,a_2,\alpha_1\}$ and $\{a_3,a_4,\alpha_1\}$ as triangles or   $\{a_1,a_4,\alpha_1\}$ and $\{a_2,a_3,\alpha_1\}$ as triangles. By symmetry, we may assume that we are in the second case. Then, by submodularity using the sets $\{b_1,b_2,a_1,a_4,b_4,\alpha_1\}$ and 
		$\{b_2,b_3,a_2,a_3,b_4,\alpha_1\}$, we deduce that $r(\{b_1,b_4,\alpha_1\}) = 2$. It follows that  
		 $M \cong M(K_5\backslash e)$, a contradiction. We conclude that \ref{wh_claim} holds.
		
		Now suppose that $W$ spans $Z$. If $M|W$ is a rank-$2$ whirl, then $M|Z \cong U_{2,5}$, a contradiction. If $M|W$ is a rank-$3$ wheel, then one easily checks that $M|Z$ is isomorphic to one of $O_7$, $F_7^-$, or $F_7$, a contradiction.
		
		We may now assume that $W$ does not span $Z$. Then $m > 1$. By (\ref{beta}),  $\beta_{m+1} \in W \cup (Q - \{\beta_{m+1}\})$. We will first suppose that $\beta_{m+1} = \beta_i$ for some $i$ in $\{1,2,\dots,m\}$. Then $\{\beta_i,\beta_{i+1},\dots,\beta_m\}$ is an independent set and 
		$\{\beta_j,\alpha_j,\beta_{j+1}\}$ is a triangle for all $j$ in $\{i,i+1,\dots,m\}$. By \ref{wh_claim} and Lemma~\ref{ww0}, for 
		$R = \{\beta_i,\alpha_i,\beta_{i+1},\alpha_{i+1},\dots,\beta_m,\alpha_m\}$, the matroid $M|R$ is a rank-$3$ wheel or a rank-$2$ whirl. 
		Then the matroid obtained from $M|Z$ by contracting   $\{\alpha_2,\alpha_3,\dots, \alpha_{i-1}\}$ and simplifying is the parallel connection of $M|W$ and $M|R$, that is, $M|Z$ has as a minor one of $P(U_{2,4},U_{2,4})$, $P(U_{2,4},M(K_4))$, and  $P(M(K_4),M(K_4))$, a contradiction.
		
		Finally, suppose that $\beta_{m+1} \not \in \{\beta_1,\beta_2,\dots,\beta_m\}$. Then $\beta_{m+1}$ is $\alpha_i$ for some $i \ge 1$, or 
		$\beta_{m+1} \in W$. 
		Consider the first case and take $\alpha_{m+1} = \beta_i$. Then, by \ref{wh_claim} and Lemma~\ref{ww0}, with $R = \{\beta_{i+1},\alpha_{i+1},\dots,\beta_{m+1},\alpha_{m+1}\}$, we have that $M|R$ is a rank-3 wheel or a rank-2 whirl. Contracting $\{\alpha_2,\alpha_3,\dots, \alpha_{i-1}\}$ from $M|Z$ and simplifying, we obtain one of $P(U_{2,4},U_{2,4})$, $P(U_{2,4},M(K_4))$, and  $P(M(K_4),M(K_4))$, a contradiction. 
 In the second case, when $\beta_{m+1} \in W$, we recall that $\beta_1 = a_1$. Suppose that $\{\beta_1,\beta_{m+1}\}$ is not in a triangle of $M|W$. Then $M|W \cong M(K_4)$ and $\beta_{m+1} = b_3$. By assumption, $\{b_1,b_2,b_3\} \cup \{\beta_2,\dots,\beta_{m}\}$ is independent. By Lemma~\ref{ww0}, the triangles $\{b_1,b_2,a_1\}$, $\{a_1,\alpha_1, \beta_2\},\dots, \{\beta_m,\alpha_m, b_3\}$, $\{b_3,a_3,b_1\}$ imply that $M|Z$ has a wheel or whirl of rank at least four as a restriction, a contradiction. We deduce  that $\{\beta_1,\beta_{m+1}\}$ is   in a triangle of $M|W$. 
 Then, by symmetry, we may assume that $\beta_{m+1} = b_1$. 
 We let $\alpha_{m+1} = b_2$. Then, for $R = \{\beta_{1},\alpha_{1},\dots,\beta_{m+1},\alpha_{m+1}\}$, we have that $M|R$ is a rank-3 wheel or a rank-2 whirl. But $\alpha_1 \not\in \cl(W)$, so $M|R$ is a rank-3 wheel. If $M|W$ is a rank-$2$ whirl, then  $O_7$ is a restriction of $M|Z$, a contradiction. If $M|W$ is a rank-$3$ wheel, then $M|(W \cup R)$ has rank four and consists of two copies of $M(K_4)$ sharing a  triangle. This matroid is $M(K_5 \backslash e)$, a contradiction.
\end{proof}


\section{The  density-critical matroids of small density}


In this section, we prove Theorem~\ref{9/4}. The following result \cite{JGO81} (see also \cite[Lemma 4.3.10]{JGO11}) will be used repeatedly in this proof. 

\begin{lemma}
\label{minc}
Let $M$ be a connected matroid having at least two elements and let $\{e_1,e_2,\dots,e_m\}$ be a cocircuit of $M$ such that $M/e_i$ is disconnected for all $i$ in $\{1,2,\dots,m-1\}$. Then $\{e_1,e_2,\dots,e_{m-1}\}$ contains a $2$-circuit of $M$.
\end{lemma}

We shall make repeated use of the following consequence of this lemma.

\begin{corollary}
\label{dir2}
Let $M$ be a simple connected matroid and $Z$ be a non-empty non-spanning subset of $E(M)$. Then $M$ has a simple connected minor N such that $N|Z = M|Z$ and $r(N) = r_M(Z)$.
\end{corollary}

\begin{proof}
Let $C^*$ be a cocircuit of $M$ that is disjoint from $\cl(Z)$. As $M$ is simple, it follows by Lemma~\ref{minc} that there is an element $e$ of $C^*$ such that $M/e$ is connected. Since $e \not \in \cl(Z)$, we see that $(M/e)|Z = M|Z$. Clearly we can label $\si(M/e)$ so that its ground set contains $Z$. If $r(M) - r(Z) = 1$, then we take $N = \si(M/e)$. Otherwise we repeat the above process using $\si(M/e)$ in place of $M$. After $r(M) - r(Z)$ applications of this process, we obtain the desired minor $N$.
\end{proof}

The next result, which was proved by Dirac~\cite{GAD60}, follows easily by induction after  recalling that a connected matroid with no minor isomorphic to $U_{2,4}$ or $M(K_4)$ is isomorphic to the cycle matroid of a series-parallel network.

\begin{lemma}
\label{dir}
Let $M$ be a simple matroid having no minor isomorphic to $U_{2,4}$ or $M(K_4)$. Then 
$$|E(M)| \le 2r(M) - 1.$$
\end{lemma}

We omit the elementary proof of the next result a consequence of which is that every density-critical matroid   is connected. 

\begin{lemma}
\label{crit}
Let $M_1$ and $M_2$ be matroids of  rank at least one. Then 
$$d(M_1 \oplus M_2) \le \max\{d(M_1), d(M_2)\}.$$
Moreover, equality holds here if and only if $d(M_1) = d(M_2)$.
\end{lemma}

The next result will be useful in identifying the density-critical matroids of density at most two.

\begin{lemma}
\label{elconn}
Let $M$ be a density-critical matroid with $d(M) \le 2$. If $(X_1,X_2)$ is a $2$-separation  of $M$, then there is an element $p$ in $\cl(X_1) \cap \cl(X_2)$, and $M = P(M|(X_1 \cup \{p\}), M|(X_2 \cup \{p\}))$. 
\end{lemma}

\begin{proof}
As $(X_1,X_2)$ is a $2$-separation of $M$,  for some element $q$ not in $E(M)$, we can write $M$ as $M_1 \oplus_2 M_2$ where each $M_i$ has ground set $X_i \cup \{q\}$. Let $|E(M_i)| = n_i$ and $r(M_i) = r_i$. Assume that both $M_1$ and $M_2$ are simple. Then $\tfrac{|E(M)|}{r(M)} > \tfrac{|E(M_1)|}{r(M_1)}$, so
$$\frac{n_1 + n_2-2}{r_1 + r_2-1} > \frac{n_1}{r_1}.$$
Hence $$r_1n_2 - 2r_1 > r_2n_1 - n_1.$$
By symmetry,
$$r_2n_1 - 2r_2 > r_1n_2 - n_2.$$
Adding the last two inequalities gives $n_1 + n_2 > 2(r_1 + r_2),$ so $n_i > 2r_i$ for some $i$. Thus $d(M_i) > 2$. Since $M$ is density-critical with density at most two, this is a contradiction. 
We conclude that $M_1$ or $M_2$, say $M_1$, is non-simple. Thus it has an element $p$ in parallel with the basepoint $q$ of the $2$-sum. Hence  
$M = P(M|(X_1 \cup \{p\}), M|(X_2 \cup \{p\}))$. 
\end{proof}



\begin{lemma}
\label{F1} 
Let $N$ be a simple connected matroid in which all but at most one element is in at least two triangles. Then $N$ has no $2$-cocircuits. Moreover, if $N$ has $\{a,b,c\}$ as a triad, then either
\begin{itemize}
\item[(i)] $\{a,b,c\}$ is contained in a $4$-point line  and $N = P(U_{2,4}, N\del \{a,b,c\})$; or 
\item[(ii)] $N$ has a triangle $\{x,y,z\}$ such that $N|\{a,b,c,x,y,z\} \cong M(K_4)$ and $N$ is the generalized parallel connection of $N|\{a,b,c,x,y,z\}$ and \linebreak $N\del \{a,b,c\}$ across the triangle $\{x,y,z\}$.
\end{itemize}
\end{lemma}

\begin{proof}
As $N$ has at most one element that is not in at least two triangles, $N$ has no $2$-cocircuits. Suppose $\{a,b,c\}$ is a triad of $N$. If $\{a,b,c\}$ is also a triangle, then $\{a,b,c\}$ is $2$-separating in $N$. Moreover, $\{a,b,c\}$ is contained in a $4$-point line $\{a,b,c,d\}$ and (i) holds.

We may now assume that $\{a,b,c\}$ is not a triad of $N$. Then, because at least two of $a, b,$ and $c$ are in at least two triangles, the hyperplane $E(N) - \{a,b,c\}$ of $N$ contains distinct elements $x$, $y$, and $z$ such that $\{a,b,z\}$, $\{a,y,c\}$, and $\{x,b,c\}$ are triangles. Now 
\begin{align*}
r(\{x,y,z\}) & \le r(E(N) - \{a,b,c\}) + r(\cl(\{a,b,c\})) - r(N)\\
& = r(N) - 1 + 3 - r(N) = 2.
\end{align*}
Thus $\{x,y,z\}$ is a triangle of $N$ and $N|\{a,b,c,x,y,z\} \cong M(K_4)$. It follows by a result of Brylawski~\cite{THB75} (see also \cite[Proposition 11.4.15]{JGO11}) that (ii) holds.
\end{proof}

\begin{corollary}
\label{F3}
Let $N$ be a simple connected matroid in which all but at most one element is in at least two triangles and $d(N) \le \tfrac{9}{4}$. If $r(N) = 2$, then $N \cong U_{2,4}$. If $r(N) = 3$, then $N \cong M(K_4)$. If $r(N) = 4$, then $N  \cong P(U_{2,4}, M(K_4)),      M(K_5 \del e)$, or $M^*(K_{3,3})$.
\end{corollary}

\begin{proof}
We omit the straightforward proof for the case when $r(N) \in \{2,3\}$. Assume $r(N) = 4$. By Lemma~\ref{F1}, $N$ has no $2$-cocircuits. Now suppose    $N$ has $\{a,b,c\}$ as a triad. If (i) of Lemma~\ref{F1} holds, then $N = P(U_{2,4}, N\del \{a,b,c\})$. By the result in the rank-$3$ case, $N\del \{a,b,c\} \cong M(K_4)$, so $N  \cong P(U_{2,4}, M(K_4))$. If, instead, (ii) of Lemma~\ref{F1} holds, then $N$ is the generalized parallel connection across a triangle $\{x,y,z\}$ of $M(K_4)$ and $N\del \{a,b,c\}$. In the latter, $E(N\del \{a,b,c,x,y,z\})$ must be a triad of $N$, so $N\del \{a,b,c\} \cong M(K_4)$. Hence $N$ is the generalized parallel connection across a triangle of two copies of $M(K_4)$, so $N \cong M(K_5\del e)$. 

We may now assume that $N$ has no triads. Then every cocircuit of $N$ has at least four elements. As $N$ certainly has a plane that contains two intersecting triangles, $\{x,f_1,g_1\}$ and $\{x,f_2,g_2\}$, we deduce that $|E(N)| \ge 9$, so $|E(N)| = 9$. Let $\{a,b,c,d\}$ be the cocircuit $E(N) - \{x,f_1,f_2,g_1,g_2\}$. Because $N$ has no plane with more than five points and has all but at most one element in two triangles, we may assume that $\{a,b,g_1\}$ and $\{a,c,g_2\}$ are triangles of $N$. Then $N\del d$ has $\{x,f_1,g_1\}, \{g_1,b,a\}, \{a,c,g_2\}, \{g_2,f_2,x\}$ as triangles. 
By Lemma~\ref{ww0}, $N\del d$ is a rank-$4$ wheel or whirl. In this matroid, $f_1$, $b$, $c$, and $f_2$ are in unique triangles. It follows that $N$ must have $\{d,f_1,c\}$ and $\{d,b,f_2\}$ as triangles. Thus $N\del d$ is a rank-$4$ wheel. Likewise, $N\del f_1$ and $N\del c$ are also rank-$4$ wheels, so $N \cong M^*(K_{3,3})$. 
\end{proof}

\begin{lemma}
\label{F3.1} 
Let $N$ be a simple  matroid of rank at least three in which every element is in at least two triangles. Suppose $ e \in E(N)$. Then 
\begin{itemize}
\item[(i)] $e$ is in a plane of $N$ having at least seven points; or 
\item[(ii)] every element of $\si(N/e)$ is in at least two triangles; or 
\item[(iii)]  $N$ has a  $U_{2,4}$-  or $M(K_4)$-restriction using $e$. 
\end{itemize}
\end{lemma}

\begin{proof}
Assume that neither (i) nor (iii) holds. We show that every element of $\si(N/e)$ is in at least two triangles.  First consider  a triangle $\{e,c_1,c_2\}$ of $N$ containing $e$.  Let $\{c_1,d_1,f_1\}$ and 
$\{c_2,d_2,f_2\}$ be triangles of $N$ where neither contains $e$. If $r(\{e,c_1,d_1,f_1,c_2,d_2,f_2\}) = 4$, then, in $\si(N/e)$, the element $c$ corresponding to $c_1$ and $c_2$ is in at least two triangles. Now suppose $r(\{e,c_1,d_1,f_1,c_2,d_2,f_2\}) = 3$. Since $N$ has no plane with more than six points, we may assume that $f_1 = f_2$. Rename this element $f$. If $\{e,d_1,d_2\}$ is not a triangle, then $\si(N/e)$ has a $4$-point line containing $c$, so $c$ is in at least two triangles of this matroid. If $\{e,d_1,d_2\}$ is   a triangle of $N$, then $N|\{e,c_1,c_2,d_1, d_2,f\}) \cong M(K_4)$, a contradiction.

Now let $f$ be an element of $N$ that is not in a triangle with $e$. Let $\{f,g_1,h_1\}$ and $\{f,g_2,h_2\}$ be triangles of $N$. Then $\si(N/e)$ has at least two triangles containing $f$ otherwise 
$N|\{e,f, g_1,g_2, h_1,h_2\}) \cong M(K_4)$,  a contradiction.
\end{proof}

Recall that $M_{18}$ is the $18$-element matroid that is obtained by attaching, via parallel connection, a copy of $M(K_4)$  at each element of an $M(K_3)$. 

\begin{lemma}
\label{F8.1} 
Let $N$ be a simple connected non-empty matroid in which every element is in a $U_{2,4}$-  or $M(K_4)$-restriction. Assume that $d(N) \le \tfrac{9}{4}$ but $d(N') < \tfrac{9}{4}$  for all proper minors $N'$ of $N$. Then $N$ is isomorphic to $U_{2,4}, M(K_4), P (U_{2,4}, M(K_4)),    P(M(K_4), M(K_4)), 
 M(K_5 \del e)$, or $M_{18}$.
\end{lemma}

\begin{proof}
Since $d(N') \le \tfrac{9}{4}$ for all minors $N'$ of $N$, we see that, in any such $N'$,  no line   has more than four points and no plane has more than six points. Next we show the following.

\begin{sublemma}
\label{no24}
If $N$ has a $4$-point line, then $N$   is isomorphic to $U_{2,4}$ or $P (U_{2,4}, M(K_4))$.
\end{sublemma}

This is immediate if $r(N) = 2$. Because $N$ has no plane with more than six points, $r(N) \neq 3$. Let $L$ be a $4$-point line of $N$ and let $Z$ be a subset of $E(N)$ not containing $L$ such that $N|Z$ is isomorphic to $U_{2,4}$ or $M(K_4)$. If $L\cap Z \neq \emptyset$, then again, since $N$ has no plane with more than six points, we deduce that $N \cong P (U_{2,4}, M(K_4))$. We may now assume that $L \cap Z = \emptyset$. If $r(L \cup Z) \le r(Z) + 1$, then $N$ has a rank-$3$ or rank-$4$ restriction of density exceeding $\tfrac{9}{4}$, a contradiction. We deduce that $r(L \cup Z) = r(Z) + 2$.

By Corollary~\ref{dir2}, $N$ has a simple connected minor $N'$ such that $N'|(L \cup Z) = N|(L \cup Z)$  and $r(N') = r(Z) + 2$. As $N'$ is connected, it has an element $x'$ that is not in the closure of $L$ or of $Z$. Then $N'/x'$ has $N|L$ and $N|Z$ as restrictions and has rank $r(Z) + 1$. Thus 
 $\si(N'/x')$ has either a plane with more than six points or has $P (U_{2,4}, M(K_4))$  as a restriction. Each possibility yields a contradiction, so \ref{no24} holds.
 
 We may now assume that every element of $N$ is in  an $M(K_4)$-restriction. We may also assume that $N$ is not isomorphic to $M(K_4)$ or $P(M(K_4), M(K_4))$. Next we show the following.
 
 \begin{sublemma}
\label{intn}
Let $X$ and $Y$ be distinct subsets of $E(N)$ such that both $N|X$ and $N|Y$ are isomorphic to $M(K_4)$. If $|X \cap Y| \ge 2$, then $N \cong M(K_5 \del e)$.
\end{sublemma}

Since $N$ has no plane with more than six points, $r(X \cup Y) > 3$. As $|X \cap Y| \ge 2$, it follows by submodularity that $r(X \cup Y) = 4$ and $r(X \cap Y) = 2$. As $d(N|(X \cup Y)) \le \tfrac{9}{4}$, we deduce that $|X \cup Y| = 9$, so $|X \cap Y| = 3$ and $N = N|(X \cup Y)$. Moreover, $N|X$ and $N|Y$ meet in a triangle $\Delta$. By Lemma~\ref{F1}, $N$ is the generalized parallel connection of $N|X$ and $N|Y$ across $\Delta$. Thus $N \cong M(K_5 \del e)$ as each of $N|X$ and $N|Y$ is isomorphic to $M(K_4)$, so \ref{intn} holds. 

We may now assume that $E(N)$ has at least three distinct subsets $X$ with $N|X \cong M(K_4)$ and that no two such subsets meet in more than one element. 

 \begin{sublemma}
\label{intn2}
$N$ does not have $P(M(K_4),M(K_4))$ as a restriction.  
\end{sublemma}

Assume that $N|X \cong  P(M(K_4),M(K_4))$ and $N|Y \cong M(K_4)$ where $Y \not \subseteq X$. Suppose $|X \cap Y| = k$ where $k \in \{1,2\}$. Then $r(X \cup Y) \le 8-k$ and $|X \cup Y| = 17 - k$, so 
$$\frac{9}{4} \ge d(N|(X \cup Y)) \ge \frac{17-k}{8-k}.$$
Simplifying we obtain the contradiction that $4 \ge 5k \ge 5$. We deduce using \ref{intn} that $|X \cap Y| = 0$. Then $r(X \cup Y) = 8$ otherwise 
$d(N|(X \cup Y)) > \tfrac{9}{4}$. 

By Corollary~\ref{dir2}, $N$ has a simple connected minor $N'$ such that $N'|(X \cup Y) = N|(X \cup Y)$ and $r(N') = 8$. As $N|(X \cup Y)$ is disconnected, $N'$ must contain an element that is not in $X \cup Y$. Hence $|E(N')| \ge 18$, so $d(N') \ge \tfrac{9}{4}$. Thus $N' = N$ and $|E(N)| = 18$, so $N$ has a single element $z$ that is not in $X \cup Y$. The $M(K_4)$-restriction of $N$ that contains $z$ is forced to have more than one element in common with $Y$ or one of the $M(K_4)$-restrictions of $N|X$. This contradiction to \ref{intn} completes the proof of \ref{intn2}.

We now know that any two $M(K_4)$-restrictions of $N$ have disjoint ground sets. Let $X$, $Y$, and $Z$ be distinct subsets of $E(N)$ such that each of $N|X$, $N|Y$, and $N|Z$ is isomorphic to $M(K_4)$. Next we show the following.

 \begin{sublemma}
\label{intn3}
$r(X \cup Y) = 6$. Moreover, $r(X \cup Y \cup Z) = 9$ unless $N \cong M_{18}$.
\end{sublemma}

As $|X \cup Y| = 12$ and $d(N|(X \cup Y)) < \tfrac{9}{4}$, we deduce that  $r(X \cup Y) = 6$. The density constraint also means that $r(X \cup Y \cup Z)  \ge 8$. Suppose $r(X \cup Y \cup Z) = 8$. Then $d(N|(X \cup Y\cup Z)) = \tfrac{9}{4}$, so $N = N|(X \cup Y \cup Z)$. Now $r(N/Z) = 5$. As $\tfrac{12}{5} > \tfrac{9}{4}$, we must have some parallel elements in $N/Z$. As $Z$ is skew to each of $X$ and $Y$, we know that $(N/Z)|X = N|X$ and $(N/Z)|Y = N|Y$. Thus there must be elements $x$ of $X$ and $y$ of $Y$ that are parallel in $N/Z$. If there is a second such parallel pair, then $r(N/Z) \le 4$, a contradiction. In $N$, we see that $r(Z \cup \{x,y\}) = 4$. Hence, in $N/x$, we obtain a $7$-point plane $Z \cup y$ unless $\{x,y,z\}$  is a  triangle of $N$ for some $z$ in $Z$. Observe that  each of $N/x$, $N/y$, and $N/z$ is disconnected, so $N$ is obtained from $M(K_3)$ by attaching a copy of $M(K_4)$ via parallel connection at each element. Thus $N \cong M_{18}$ and \ref{intn3} holds.

By Corollary~\ref{dir2}, $N$ has a simple connected minor $N'$ of rank $9$ such that $N'|(X \cup Y \cup Z) = N|(X \cup Y \cup Z)$. As $N'$ is connected, there is an element $g$ of $E(N') - (X \cup Y \cup Z)$. Since $N'$ has no plane with more than six points, $g$ is not in the closure of any of $X$, $Y$, or $Z$ in $N'$. As  $N'/g$ has rank $8$ but has density less than $\tfrac{9}{4}$, the eighteen elements of $X \cup Y \cup Z$ cannot all be in distinct parallel classes of $N'/g$. Thus $N'$ has a triangle $\{x,y,g\}$ where we may assume that $x \in X$ and $y \in Y$. Since $N'|(X \cup Y \cup Z \cup g)$ has $Z$ as a component, there is an element $h$ of $E(N')$ that is in neither $\cl_{N'}(X \cup Y)$ nor $\cl_{N'}(Z)$. As above, $N'$ has a triangle $\{h,z,t\}$ where $t \in X \cup Y$ and $z \in Z$. Contracting $g$ and $h$ from $N'|(X \cup Y \cup Z \cup \{g,h\})$ and simplifying, we get a rank-$7$ matroid with $16$ elements. As $\tfrac{16}{7} > \tfrac{9}{4}$, we have a contradiction that completes the proof of Lemma~\ref{F8.1}. 
\end{proof}

\begin{lemma}
\label{F4} 
Let $N$ be a simple connected  matroid having an element $z$  such that each of $N$ and $\si(N/z)$ has every element in at least two triangles.   If $d(N) \le \tfrac{9}{4}$ and $d(N') < \tfrac{9}{4}$  for all proper minors $N'$ of $N$, then $N$ is isomorphic to $ P(U_{2,4}, M(K_4))$,      $M(K_5 \del e)$, or $M^*(K_{3,3})$.
\end{lemma}

\begin{proof}
We argue by induction on $r(N)$, which must be at least three. Suppose it is exactly three. Since $\si(N/z)$ has density less than $\tfrac{9}{4}$, it is isomorphic to $U_{2,4}$. As $d(N) \le \tfrac{9}{4}$, we see that $|E(N)| \le 6$. By Lemma~\ref{F1}, $N$ has no $2$-cocircuits. Thus $N$ has a triangle whose complement is a triad. By Lemma~\ref{F1} again, $N \cong M(K_4)$ and we get a contradiction. Hence $r(N) \ge 4$. If $r(N)  =  4$, then, by Corollary~\ref{F3}, $N$ is isomorphic to $ P(U_{2,4}, M(K_4)),      M(K_5 \del e)$, or $M^*(K_{3,3})$.

Now assume the result holds for $r(N) < k$ and let $r(N) = k \ge 5$. Let $N_1 = \si(N/z)$. Every element of $N_1$ is in at least two triangles. Let $N_2$ be a component of $N_1$. 
By Lemma~\ref{F3.1}, either every element of $N_2$ is in a $U_{2,4}$- or $M(K_4)$-restriction, or $N_2$ has an element $z_2$ such that every element of $\si(N_2/z_2)$ is in at least two triangles. If the latter occurs, then, by the induction assumption, $N_2$ is isomorphic to  $P (U_{2,4}, M(K_4)),      M(K_5 \del e)$, or $M^*(K_{3,3})$. Each of these matroids has density $\tfrac{9}{4}$, a contradiction. Thus 
every element of $N_2$ is in a $U_{2,4}$- or $M(K_4)$-restriction.  As $d(N_2) < \tfrac{9}{4}$, Lemma~\ref{F8.1} implies  that $N_2$, and hence each component of $N_1$, is isomorphic to one of $U_{2,4}$, $M(K_4)$, or $P(M(K_4),M(K_4))$. 

Suppose that $N_2 = N_1$. Then, as $r(N) \ge 5$, we deduce that 
 $N_1 \cong P(M(K_4),M(K_4))$. As $N_1 = \si(N/z)$, we see that $r(N) = 6$. Because $d(N) \le \tfrac{9}{4}$, it follows that $|E(N)| \le 13$. Since $z$ is in at least two triangles of $N$, we deduce that $|E(N)| \ge |E(N_1)| + 3 = 14$, a contradiction.
 
We may now assume that $N_1$ has more than one component. Hence, for some $k \ge 2$, there is a collection $N^1,N^2,\dots,N^k$ of connected matroids such that $E(N^i) \cap E(N^j) = \{z\}$ for all $i \neq j$, the matroid $N^i / z$ is connected for all $i$, and $N$ is the parallel connection of $N^1,N^2,\dots,N^k$ across the common basepoint $z$. As noted above, each $\si(N^i/z)$ is isomorphic to one of $U_{2,4}$, $M(K_4)$, or $P(M(K_4),M(K_4))$.  As every element of $N$ is in at least two triangles,   every element of each $N^i$ except possibly $z$ is in at least two triangles of $N^i$. Thus, by Corollary~\ref{F3}, $N^i \cong M(K_4)$; or $r(N^i) = 4$ and $|E(N^i)| = 9$; or $r(N^i) > 4$. In the first case, $\si(N^i/z) \not \cong U_{2,4}$; in the second case, $d(N^i) = \tfrac{9}{4}$. Both of these possibilities give contradictions, so $\si(N^i/z) \cong  P(M(K_4),M(K_4))$ for each $i$. As $z$ is in at least two triangles of $N$, we may assume the elements of two such  triangles lie in $E(N^1) \cup E(N^2)$. As $|E(\si(N^i/z))| = 11$ and $r(N^i/z) = 5$, we  see that $|E(N^1) \cup E(N^2)| \ge 25$ and $r(E(N^1) \cup E(N^2)) = 11$. But $\tfrac{25}{11} > \tfrac{9}{4}$,   a contradiction.
\end{proof}

We conclude the paper by proving Theorem~\ref{9/4}. In this proof, we will make extensive use of the Cunningham-Edmonds canonical tree decomposition of a connected matroid. The definition and properties of this decomposition may be found in \cite[Section 8.3]{JGO11}. In brief, associated with each connected matroid $M$, there is a tree $T$ that is unique up to the labelling of its edges. Each vertex of $T$ is labelled by a circuit, a cocircuit, or a $3$-connected matroid with at least four elements. Moreover,    no two adjacent vertices of $T$ are labelled by   circuits and no two adjacent vertices are labelled by cocircuits. For an edge $e$ of $T$ whose endpoints are labelled by matroids $M_1$ and $M_2$, the ground sets of these two matroids meet in $\{e\}$. When we contract $e$ from $T$, the composite vertex that results by identifying the endpoints of $e$ is labelled by the $2$-sum of $M_1$ and $M_2$. By repeating this process, contracting all of the remaining edges of $T$ one by one, we eventually obtain a single-vertex tree. Its vertex is labelled by $M$. 

Each edge $f$ of $T$  induces a partition of $E(M)$. This partition is a $2$-separation of $M$ {\it displayed by $f$}. The remaining $2$-separations of $M$ coincide with those  that are displayed by those vertices of $T$ that are labelled by circuits or cocircuits. For such a vertex $v$ having label $N$, there is a partition $\{X_1,X_2,\dots,X_k\}$ of $E(M) - E(N)$ induced by the components of $T - v$. A partition $(X,Y)$ of $E(M)$ is {\it displayed by the vertex $v$} if each $X_i$ is contained in $X$ or $Y$. Every such partition of $E(M)$ with both $X$ and $Y$ having at least two elements is a $2$-separation of $M$ and these $2$-separations along with those displayed by the edges of $T$ are all of the $2$-separations of $M$. Recall that, for all $n \ge 2$, we denote by $P_n$   any matroid that can be constructed from $n$ copies of $M(K_3)$ via a sequence of parallel connections.

\begin{proof}[Proof of Theorem~\ref{9/4}.]
Let $M$ be a density-critical matroid with $d(M) \le \tfrac{9}{4}$. Suppose   $d(M) \ge 2$. By Lemma~\ref{lem0}, every element of $M$ is in at least two triangles. By Corollary~\ref{F3}, if $r(M) \in \{2,3\}$, then $M$ is   $U_{2,4}$ or $M(K_4)$. We may now assume that $r(M) \ge 4$. By Lemma~\ref{F3.1}, either every element of $M$ is in a $U_{2,4}$- or $M(K_4)$-restriction, or, for some element $z$ of $M$,  every element of $\si(M/z)$ is in at least two triangles. In the first case, by Lemma~\ref{F8.1}, $M$ is isomorphic to  $P (U_{2,4}, M(K_4)),    
P(M(K_4), M(K_4)), 
 M(K_5 \del e)$, or $M_{18}$. In the second case, by Lemma~\ref{F4}, $M$ is isomorphic to $P(U_{2,4}, M(K_4)),      M(K_5 \del e)$, or $M^*(K_{3,3})$. 
 Thus the theorem identifies all possible density-critical matroids with density in $[2,\tfrac{9}{4}]$ and one easily checks that each of the matroids identified is indeed density-critical. 
 
 Now suppose that $d(M) < 2$. By Lemma~\ref{crit}, $M$ is connected. Clearly, if $r(M)$ is $1$ or $2$, then $M$ is isomorphic to $U_{1,1}$ or $U_{2,3}$. As $U_{2,4}$ and  $M(K_4)$   both have density $2$, $M$ is a series-parallel network (see, for example, \cite[Corollary 12.2.14]{JGO11}). Thus, in the Cunningham-Edmonds canonical tree decomposition $T$ of $M$, every vertex is labelled by a circuit or a cocircuit. Since $M$ is simple, for every vertex of $T$ that is labelled by a cocircuit $C^*$, at most one element of $C^*$ is in $E(M)$. Let $e$ be an edge of $T$ that meets the vertex labelled by $C^*$. Then, for the $2$-separation $(X,Y)$ of $M$ that is displayed by $e$,   Lemma~\ref{elconn} implies that   $M$ has an element $p$   in $\cl(X) \cap \cl(Y)$. Thus $p \in C^*$, so $C^*$ contains exactly one element of $M$.
 
 Now take a vertex of $T$ that is labelled by a circuit $C$ where $C = \{e_1,e_2,\dots,e_k\}$ and suppose that $k \ge 4$. Suppose $e_1 \in E(M)$. Then $M/e_1$ is simple having rank $r(M) - 1$. As 
  $\tfrac{|E(M)| - 1}{r(M) - 1} <  \tfrac{|E(M)|}{r(M)},$  
 we obtain the contradiction that $|E(M)| < r(M)$. We deduce that $C \cap E(M) = \emptyset$. Now $T\del e_1,e_2$ has exactly three components. Let $T'$ be the one containing $e_3$ and let $X$ be the subset of $E(M)$ corresponding to $T'$. Then $(X,E(M) - X)$ is a $2$-separation of $M$. By Lemma~\ref{elconn}, there is an element $p$ of $M$ that is in $\cl(X) \cap \cl(E(M) - X)$. But the tree decomposition implies that there is no such element. We deduce that $C$ has exactly three elements. Thus every vertex of $T$ that is labelled by a circuit  is labelled by a triangle. Since every vertex of $T$ that is labelled by a cocircuit has exactly one element of $E(M)$ in that cocircuit, a straightforward induction argument establishes that, for some $n \ge 2$, the matroid $M$ is obtained from $n$ copies of $M(K_3)$ by a sequence of $n-1$ parallel connections. Thus $M \cong P_n$. 
 
 Finally, we show by induction  that $P_n$ is density-critical. This is true for $n = 1$. Assume it true for $n < m$ and let $n = m\ge 2$. Take $x$ in $P_n$.  
 Assume first that $x$ is in exactly one triangle $\{x,y,z\}$. Then $\si(P_n/x) \cong P_n/x\del z$. As the last matroid is easily seen to be isomorphic to the density-critical matroid $P_{n-1}$ and $d(P_{n-1}) < d(P_n)$, we deduce that every minor of $P_n/x$ has density less that $d(P_n)$. Now assume that $x$ is in at least two triangles of $P_n$. Then $\si(P_n/x)$ is easily seen to be the direct sum of a collection of matroids each of which is isomorphic to some $P_k$ with $k < n$. By Lemma~\ref{crit} and the induction assumption,  every minor of $P_n/x$ has density less that $d(P_n)$. We conclude that $P_n$ is density-critical, so the theorem is proved.
\end{proof}

\section*{Acknowledgements}
The work  for this paper was done at the Tutte Centenary Retreat (https://www.matrix-inst.org.au/events/tutte-centenary-retreat/) held at the MAThematical Research Institute (MATRIx), Creswick, Victoria, Australia, 26 November -- 2 December, 2017. The authors are very grateful for the support provided by MATRIx. The first, second, and fourth  authors were supported, respectively, by NSERC Scholarship PGSD3-489418-2016,  National Science Foundation Grant 1500343, and the New Zealand Marsden Fund.


\begin{thebibliography}{88}

\bibitem{THB75} T.H.\ Brylawski, Modular constructions for combinatorial geometries, Trans. Amer. Math. Soc. 203 (1975), 1--44.

\bibitem{GAD60} G.A.\ Dirac, In Abstrakten Graphen vorhande vollst\"andige $4$-Graphen und ihre Unterteilungen, Math. Nach. 22 (1960), 61--85.

\bibitem{JE65} J.\  Edmonds, Minimum partition of a matroid into independent sets, J. Res. Nat. Bur. Standards Sect. B 69B (1965), 67--72.

\bibitem{GGW} J.\ Geelen, B.\ Gerards, G.\ Whittle, Solving Rota's Conjecture, Notices Amer. Math. Soc. 61 (2014), 736--743.

\bibitem{DWH43} D.W.\ Hall, A note on primitive skew curves, Bull. Amer. Math. Soc. 49 (1943), 935--937.

\bibitem{JGO81} J.G.\ Oxley, On connectivity in matroids and graphs, Trans. Amer. Math. Soc. 265 (1981), 47--58.

\bibitem{JGO11} J.\ Oxley, Matroid Theory, Second edition, Oxford Univ.\ Press, New York, 2011.

\bibitem{PDS80} P.D.\ Seymour, Decomposition of regular matroids, J. Combin. Theory Ser. B  28 (1980), 305--359.

\bibitem{PDS85} P.D.\ Seymour, Minors of $3$-connected  matroids, European J. Combin.   6 (1985), 375--382.

\end{thebibliography}
\end{document}